\documentclass{amsart}
\usepackage{graphicx}
\usepackage{easywncy}
\newtheorem{thm}{Theorem}[section]
\newtheorem{lem}[thm]{Lemma}
\newtheorem{cor}[thm]{Corollary}
\newtheorem{prop}[thm]{Proposition}
\theoremstyle{definition}
\newtheorem{defn}[thm]{Definition}
\theoremstyle{remark}
\newtheorem{ex}[thm]{Example}
\newtheorem{rem}[thm]{Remark}
\newcommand{\sha}{\mathbin{\widetilde{\mathcyr{sh}}}}
\newcommand{\bigsha}{\mathop{\widetilde{\raisebox{-5pt}{\scalebox{3}{$\mathcyr{sh}$}}}}}
\newcommand{\sym}{\mathfrak{S}}
\renewcommand{\bar}{\overline}
\newcommand{\A}{\mathcal{A}}
\newcommand{\G}{\mathrm{G}}
\renewcommand{\H}{\mathfrak{H}}
\newcommand{\I}{\mathcal{I}}

\newcommand{\M}{\mathcal{M}}
\newcommand{\N}{\mathcal{N}}
\renewcommand{\P}{\mathcal{P}}
\newcommand{\Q}{\mathbb{Q}}
\newcommand{\Z}{\mathbb{Z}}
\renewcommand{\Im}{\operatorname{Im}}
\newcommand{\odd}{\mathrm{odd}}
\newcommand{\even}{\mathrm{even}}

\newcommand{\ve}{\boldsymbol}
\DeclareMathOperator{\dep}{dp}
\DeclareMathOperator{\wt}{wt}

\DeclareMathOperator{\id}{id}
\DeclareMathOperator{\chara}{char}
\begin{document}
\title{The Bowman-Bradley theorem for multiple zeta-star values}
\author{Hiroki Kondo}
\address{Nisshin Fire \& Marine Insurance Company, Limited,
 2--3 Kanda-Surugadai, Chiyoda-ku, Tokyo, 101--8329, Japan.}
\email{hiroki.kondou@nisshinfire.co.jp}
\author{Shingo Saito}
\address{Faculty of Mathematics, Kyushu University, 744, Motooka, Nishi-ku, Fukuoka, 819--0395, Japan.}
\email{ssaito@math.kyushu-u.ac.jp}
\author{Tatsushi Tanaka}
\address{Faculty of Mathematics, Kyushu University, 744, Motooka, Nishi-ku, Fukuoka, 819--0395, Japan.}
\email{t.tanaka@math.kyushu-u.ac.jp}
\begin{abstract}
 The Bowman-Bradley theorem asserts that
 the multiple zeta values at the sequences obtained by
 inserting a fixed number of twos between $3,1,\ldots,3,1$
 add up to a rational multiple of a power of $\pi$.
 We establish its counterpart for multiple zeta-star values
 by showing an identity in a non-commutative polynomial algebra introduced by Hoffman.
\end{abstract}
\maketitle
\section{Introduction}
For $k_1,\ldots,k_n\in\Z_{\ge1}$ with $k_1\ge2$,
the multiple zeta value (MZV) and the multiple zeta-star value (MZSV)
are defined by
\[
 \zeta(k_1,\ldots ,k_n)=\sum_{m_1>\cdots >m_n>0}\frac{1}{m_1^{k_1}\cdots m_n^{k_n}},\quad
 \zeta^{\star}(k_1,\ldots ,k_n)=\sum_{m_1\ge \cdots \ge m_n>0}\frac{1}{m_1^{k_1}\cdots m_n^{k_n}},
\]
respectively.
When $n=1$, the MZVs and MZSVs coincide and reduce to the values of the Riemann zeta function
at positive integers.
Euler~\cite{E} found that the Riemann zeta values at even positive integers
are rational multiples of powers of $\pi^2$:
\[
  \zeta(2k)=-\frac{B_{2k}(2\pi\sqrt{-1})^{2k}}{2(2k)!},
\]
where the rational numbers $B_{2k}$ are the Bernoulli numbers given by
\[
 \sum_{m=0}^{\infty}B_{m}\frac{t^m}{m!}=\frac{te^t}{e^t-1}.
\]
This result has been generalized to MZVs
(Hoffman~\cite{Hof}, Ohno-Zagier~\cite{OZ}, Yamasaki~\cite{Y}, etc.) and
MZSVs (Hoffman~\cite{Hof}, Aoki-Kombu-Ohno~\cite{AKO}, Zlobin~\cite{Zlo}, Muneta~\cite{Mun}, etc.):
for $k,n\in\Z_{\ge1}$, we have
\[
 \zeta(\{2k\}^n),\zeta^{\star}(\{2k\}^n)\in\Q\pi^{2kn}.
\]
Here and throughout, we write
\[
 \{k_1,\ldots,k_l\}^m=\underbrace{k_1,\ldots,k_l,k_1,\ldots,k_l,\ldots,k_1,\ldots,k_l}_{lm}.
\]
It has also been shown that
\[
 \zeta(\{3,1\}^{n}),\zeta^{\star}(\{3,1\}^{n})\in\Q\pi^{4n}
\]
for $n\in\Z_{\ge1}$,
by Borwein-Bradley-Broadhurst-Lison\v{e}k~\cite{BBBL} and Kontsevich-Zagier~\cite{KonZag} for MZVs
and by Zlobin~\cite{Zlo} and Muneta~\cite{Mun} for MZSVs.
Furthermore, Bowman-Bradley~\cite{BowBra} obtained
\[
 \sum_{\substack{\sum_{i=0}^{2n}m_i=m\\m_0,\ldots,m_{2n}\ge 0}}
 \zeta(\{2\}^{m_0},3,\{2\}^{m_1},1,\{2\}^{m_2},\ldots,3,\{2\}^{m_{2n-1}},1,\{2\}^{m_{2n}})
 \in\Q\pi^{2m+4n}
\]
for $m\in\Z_{\ge0}$ and $n\in\Z_{\ge1}$.
However, its counterpart for MZSVs has been proved only partially:
the $m=1$ case by Muneta~\cite{Mun} and $m=2$ by Imatomi-Tanaka-Tasaka-Wakabayashi~\cite{ITTW}.
In the present paper, we establish the complete counterpart by showing that
\[
 \sum_{\substack{\sum_{i=0}^{2n}m_i=m\\m_0,\ldots,m_{2n}\ge 0}}
 \zeta^{\star}(\{2\}^{m_0},3,\{2\}^{m_1},1,\{2\}^{m_2},\ldots,3,\{2\}^{m_{2n-1}},1,\{2\}^{m_{2n}})
 \in\Q\pi^{2m+4n}
\]
for all $m\in\Z_{\ge0}$ and $n\in\Z_{\ge1}$.
Note that this is valid when $m\in\Z_{\ge1}$ and $n=0$
because $\zeta^{\star}(\{2\}^m)\in\Q\pi^{2m}$ as mentioned above,
and also clearly valid when $m=n=0$ because of the convention that
the value of $\zeta^{\star}$ at the empty sequence $\emptyset$ is $1$.
Therefore our main theorem reads as follows:
\begin{thm}[Main Theorem]\label{thm:main}
 For all $m,n\in\Z_{\ge0}$, we have
 \[
  \sum_{\substack{\sum_{i=0}^{2n}m_i=m\\m_0,\ldots,m_{2n}\ge 0}}
  \zeta^{\star}(\{2\}^{m_0},3,\{2\}^{m_1},1,\{2\}^{m_2},\ldots,3,\{2\}^{m_{2n-1}},1,\{2\}^{m_{2n}})
  \in\Q\pi^{2m+4n}.
 \]
\end{thm}

\section{Reduction to an algebraic identity}
\subsection{Hoffman's algebraic setup}
We first reduce our main theorem to an identity in the algebra introduced by Hoffman~\cite{Hof2}.
Our definitions of operations are superficially different from those given by Hoffman,
as we intend to define similar operations on another algebra later in a unified manner
that is convenient to prove the key algebraic identity.
It will be seen that our definitions are essentially the same as Hoffman's.

Let $\H=\Q\langle x,y \rangle$ be the non-commutative polynomial algebra
over the rational numbers in two indeterminates $x$ and $y$,
and denote the subalgebras $\Q+\H y$ and $\Q+x\H y$ of $\H$ by $\H^1$ and $\H^0$, respectively.
Write $z_{a}=x^{a-1}y\in\H^1$ for $a\in\Z_{\ge1}$.

\begin{defn}
 An \emph{index} is a finite sequence of positive integers;
 the empty sequence $\emptyset$ is also regarded as an index.
 An index is \emph{admissible} if
 either it is empty or its first component is greater than $1$.
 The sets of all indices and all admissible indices are denoted by $I$ and $I_0$, respectively.
\end{defn}

For $\ve{a}=(a_1,\ldots,a_n)\in I$,
put $z_{\ve{a}}=z_{a_1}\cdots z_{a_n}\in\H^1$,
where $z_{\emptyset}=1$.
Then it is easy to observe that $\{z_{\ve{a}}\mid\ve{a}\in I\}$
and $\{z_{\ve{a}}\mid\ve{a}\in I_0\}$
are $\Q$-vector space bases for $\H^1$ and $\H^0$, respectively.

\begin{defn}
 We define $\Q$-linear maps $Z,\bar{Z}\colon\H^0 \to \mathbb{R}$ by setting
 $Z(z_{\ve{a}})=\zeta(\ve{a})$ and
 $\bar{Z}(z_{\ve{a}})=\zeta^{\star}(\ve{a})$
 for $\ve{a}\in I_0$,
 where $\zeta(\emptyset)=\zeta^{\star}(\emptyset)=1$.
\end{defn}

For $n\in\Z_{\ge0}$, set $[n]=\{1,\ldots,n\}$, where $[0]=\emptyset$.
Recall that for every $l\in\Z_{\ge0}$,
there exists a unique map from $[0]$ to $[l]$, denoted by $\emptyset$,
and it is strictly increasing and has image $[0]$.
If $n\in\Z_{\ge1}$, then there is no map from $[n]$ to $[0]$.
A map $\sigma\colon[n]\to[l]$ is denoted by
\[
 \begin{pmatrix}1&\cdots&n\\\sigma(1)&\cdots&\sigma(n)\end{pmatrix}.
\]

\subsubsection{The transformation $d$}
It is known that if we define a $\Q$-linear transformation $d$ on $\H^1$
by setting $d(1)=1$ and $d(wy)=\varphi(w)y$ for $w\in\H$,
where $\varphi$ is an automorphism on $\H$ satisfying
$\varphi(1)=1$, $\varphi(x)=x$ and $\varphi(y)=x+y$,
then we have $\bar{Z}=Z\circ d$ on $\H^0$.
Here we give an alternative equivalent definition of $d$:

\begin{defn}\label{defn:d_in_H^1}
 For $n\in\Z_{\ge0}$, we write $S_n^d$ for the set of all pairs $(l,\sigma)$ of $l\in\Z_{\ge0}$
 and $\sigma\colon[n]\to[l]$ that is nondecreasing and surjective.
 We define a $\Q$-linear transformation $d$ on $\H^1$ by setting
 \[
  d(z_{\ve{a}})=\sum_{(l,\sigma)\in S_n^d}z_{\ve{c}}
 \]
 for $\ve{a}=(a_1,\ldots,a_n)\in I$,
 where $\ve{c}=(c_1,\ldots,c_l)\in I$ is given by
 \[
  c_i=\sum_{s\in\sigma^{-1}(i)}a_s
 \]
 for $i\in[l]$.
\end{defn}

\begin{rem}
 In Definition~\ref{defn:d_in_H^1}, we suppressed the explicit dependence
 of $\ve{c}$ on $\ve{a}$, $l$, $\sigma$ in the notation for simplicity.
 We will occasionally employ a similar abuse of notation in what follows
 when there is no ambiguity.
\end{rem}

\begin{ex}
 For $n=0$, we have $S_0^d=\{(0,\emptyset)\}$ and so $d(1)=1$.

 For $n=1$, we have
 \[
  S_1^d=\Biggl\{\Biggl(1,\begin{pmatrix}1\\1\end{pmatrix}\Biggr)\Biggr\}
 \]
 and so
 \[
  d(z_{a_1})=z_{a_1}
 \]
 for $a_1\in\Z_{\ge1}$.

 For $n=2$, we have
 \[
  S_2^d=\Biggl\{\Biggl(1,\begin{pmatrix}1&2\\1&1\end{pmatrix}\Biggr),
         \Biggl(2,\begin{pmatrix}1&2\\1&2\end{pmatrix}\Biggr)\Biggr\}
 \]
 and so
 \[
  d(z_{a_1}z_{a_2})=z_{a_1+a_2}+z_{a_1}z_{a_2}
 \]
 for $a_1,a_2\in\Z_{\ge1}$.

 For $n=3$, we have
 \[
  S_3^d=\Biggl\{\Biggl(1,\begin{pmatrix}1&2&3\\1&1&1\end{pmatrix}\Biggr),
         \Biggl(2,\begin{pmatrix}1&2&3\\1&1&2\end{pmatrix}\Biggr),
         \Biggl(2,\begin{pmatrix}1&2&3\\1&2&2\end{pmatrix}\Biggr),
         \Biggl(3,\begin{pmatrix}1&2&3\\1&2&3\end{pmatrix}\Biggr)\Biggr\}
 \]
 and so
 \[
  d(z_{a_1}z_{a_2}z_{a_3})
  =z_{a_1+a_2+a_3}+z_{a_1+a_2}z_{a_3}
   +z_{a_1}z_{a_2+a_3}+z_{a_1}z_{a_2}z_{a_3}
 \]
 for $a_1,a_2,a_3\in\Z_{\ge1}$.
\end{ex}

\begin{rem}
 If $(l,\sigma)\in S_n^d$, then $l\le n$.
\end{rem}

\begin{prop}
 The transformation $d$ defined in Definition~\ref{defn:d_in_H^1} is the same as the one
 defined immediately before Definition~\ref{defn:d_in_H^1}.
 It follows that $\bar{Z}=Z\circ d$ on $\H^0$.
\end{prop}

\begin{proof}
 It suffices to prove that
 \[
  d(z_{a_1}\cdots z_{a_n})=x^{a_1-1}(x+y)\cdots x^{a_{n-1}-1}(x+y)x^{a_n-1}y\tag{$\natural$}
 \]
 if $n\ge2$ and $a_1,\ldots,a_n\in\Z_{\ge1}$.
 For each $(l,\sigma)\in S_n^d$,
 setting $c_i=\sum_{s\in\sigma^{-1}(i)}a_s$ for $i\in[l]$,
 we have
 \[
  z_{c_1}\cdots z_{c_l}=x^{a_1-1}u_1\cdots x^{a_{n-1}-1}u_{n-1}x^{a_n-1}y,
 \]
 where $u_s=x$ if $\sigma(s)=\sigma(s+1)$ and $u_s=y$ if $\sigma(s)+1=\sigma(s+1)$
 for $s\in[n-1]$.
 This gives a one-to-one correspondence between the elements of $S_n^d$ and
 the terms that appear in the expansion of the right-hand side of ($\natural$).
\end{proof}

\subsubsection{The binary operations $*$ and $\sha$}
Hoffman~\cite{Hof} defined a $\Q$-bilinear map $*\colon\H^1\times\H^1\to\H^1$,
called the harmonic product,
by setting $w*1=1*w=w$ and $z_aw*z_bw'=z_a(w*z_bw')+z_b(z_aw*w')+z_{a+b}(w*w')$
for $w,w'\in\H^1$, whereas
Muneta~\cite{Mun2} defined a $\Q$-bilinear map $\sha\colon\H^1\times\H^1\to\H^1$
by setting $w\sha1=1\sha w=w$ and $z_aw\sha z_bw'=z_a(w\sha z_bw')+z_b(z_aw\sha w')$
for $w,w'\in\H^1$.
Here we give alternative equivalent definitions of $*$ and $\sha$:
\begin{defn}
 Let $m,n\in\Z_{\ge0}$.
 We write $S_{m,n}^{*}$ for the set of all triples $(l,\sigma,\tau)$ of
 $l\in\Z_{\ge0}$, $\sigma\colon[m]\to[l]$ and $\tau\colon[n]\to[l]$ such that
 $\sigma$ and $\tau$ are strictly increasing and satisfy $\Im\sigma\cup\Im\tau=[l]$.
 We also write $S_{m,n}^{\sha}$ for the set of those $(l,\sigma,\tau)\in S_{m,n}^{*}$ for which
 $\Im\sigma\cap\Im\tau=\emptyset$.

 Define $\Q$-bilinear maps $*,\sha\colon\H^1\times\H^1\to\H^1$ by setting
 \[
  z_{\ve{a}}*z_{\ve{b}}=\sum_{(l,\sigma,\tau)\in S_{m,n}^{*}}z_{\ve{c}},\qquad
  z_{\ve{a}}\sha z_{\ve{b}}=\sum_{(l,\sigma,\tau)\in S_{m,n}^{\sha}}z_{\ve{c}}
 \]
 for $\ve{a}=(a_1,\ldots,a_m),\ve{b}=(b_1,\ldots,b_n)\in I$,
 where $\ve{c}=(c_1,\ldots,c_l)\in I$ is given by
 \[
  c_i=\sum_{s\in\sigma^{-1}(i)}a_s+\sum_{t\in\tau^{-1}(i)}b_t
 \]
 for $i\in[l]$.
\end{defn}

\begin{ex}
 For $m\in\Z_{\ge0}$ and $n=0$,
 we have $S_{m,0}^*=S_{m,0}^{\sha}=\{(m,\id_{[m]},\emptyset)\}$ and so
 \[
  z_{\ve{a}}*1=z_{\ve{a}}\sha1=z_{\ve{a}}
 \]
 for all $\ve{a}\in I$.

 For $m=n=1$, we have
 \begin{align*}
  S_{1,1}^*&=\Biggl\{\Biggl(1,\begin{pmatrix}1\\1\end{pmatrix},\begin{pmatrix}1\\1\end{pmatrix}\Biggr),
             \Biggl(2,\begin{pmatrix}1\\1\end{pmatrix},\begin{pmatrix}1\\2\end{pmatrix}\Biggr),
             \Biggl(2,\begin{pmatrix}1\\2\end{pmatrix},\begin{pmatrix}1\\1\end{pmatrix}\Biggr)\Biggr\},\\
  S_{1,1}^{\sha}
   &=\Biggl\{\phantom{\Biggl(1,\begin{pmatrix}1\\1\end{pmatrix},\begin{pmatrix}1\\1\end{pmatrix}\Biggr),{}}
     \Biggl(2,\begin{pmatrix}1\\1\end{pmatrix},\begin{pmatrix}1\\2\end{pmatrix}\Biggr),
     \Biggl(2,\begin{pmatrix}1\\2\end{pmatrix},\begin{pmatrix}1\\1\end{pmatrix}\Biggr)\Biggr\}
 \end{align*}
 and so
 \[
  z_{a_1}*z_{b_1}=z_{a_1+b_1}+z_{a_1}z_{b_1}+z_{b_1}z_{a_1},\qquad
  z_{a_1}\sha z_{b_1}=z_{a_1}z_{b_1}+z_{b_1}z_{a_1}
 \]
 for $a_1,b_1\in\Z_{\ge1}$.

 For $m=2$ and $n=1$, we have
 \begin{align*}
  S_{2,1}^*
   &=\Biggl\{\Biggl(2,\begin{pmatrix}1&2\\1&2\end{pmatrix},\begin{pmatrix}1\\1\end{pmatrix}\Biggr),
     \Biggl(2,\begin{pmatrix}1&2\\1&2\end{pmatrix},\begin{pmatrix}1\\2\end{pmatrix}\Biggr),\\
   &\phantom{{}=\Biggl\{}\Biggl(3,\begin{pmatrix}1&2\\1&2\end{pmatrix},\begin{pmatrix}1\\3\end{pmatrix}\Biggr),
     \Biggl(3,\begin{pmatrix}1&2\\1&3\end{pmatrix},\begin{pmatrix}1\\2\end{pmatrix}\Biggr),
     \Biggl(3,\begin{pmatrix}1&2\\2&3\end{pmatrix},\begin{pmatrix}1\\1\end{pmatrix}\Biggr)\Biggr\},\\
  S_{2,1}^{\sha}
   &=\Biggl\{\Biggl(3,\begin{pmatrix}1&2\\1&2\end{pmatrix},\begin{pmatrix}1\\3\end{pmatrix}\Biggr),
     \Biggl(3,\begin{pmatrix}1&2\\1&3\end{pmatrix},\begin{pmatrix}1\\2\end{pmatrix}\Biggr),
     \Biggl(3,\begin{pmatrix}1&2\\2&3\end{pmatrix},\begin{pmatrix}1\\1\end{pmatrix}\Biggr)\Biggr\}
 \end{align*}
 and so
 \begin{align*}
  z_{a_1}z_{a_2}*z_{b_1}
  &=z_{a_1+b_1}z_{a_2}+z_{a_1}z_{a_2+b_1}+z_{a_1}z_{a_2}z_{b_1}
    +z_{a_1}z_{b_1}z_{a_2}+z_{b_1}z_{a_1}z_{a_2},\\
  z_{a_1}z_{a_2}\sha z_{b_1}
  &=\phantom{z_{a_1+b_1}z_{a_2}+z_{a_1}z_{a_2+b_1}+{}}
    z_{a_1}z_{a_2}z_{b_1}+z_{a_1}z_{b_1}z_{a_2}+z_{b_1}z_{a_1}z_{a_2}
 \end{align*}
 for $a_1,a_2,b_1\in\Z_{\ge1}$.
\end{ex}

\begin{rem}
 If $(l,\sigma,\tau)\in S_{m,n}^{*}$, then $\max\{m,n\}\le l\le m+n$;
 if $(l,\sigma,\tau)\in S_{m,n}^{\sha}$, then $l=m+n$.
\end{rem}

\newlength{\lensigma}
\settowidth{\lensigma}%
 {$\begin{pmatrix}1&\cdots&m&m+1\\\sigma(1)&\cdots&\sigma(m)&l+1\end{pmatrix}$}
\newlength{\lentau}
\settowidth{\lentau}%
 {$\begin{pmatrix}1&\cdots&n&n+1\\\tau(1)&\cdots&\tau(n)&l+1\end{pmatrix}$}

\begin{prop}
 The map $*$ defined above is the same as the one defined by Hoffman~\cite{Hof}.
 It follows that $*$ is an associative and commutative product with respect to which
 $Z$ is homomorphic, i.e., $Z(w*w')=Z(w)Z(w')$ for all $w,w'\in\H^0$.
\end{prop}

\begin{proof}
 It suffices to prove that
 \[
  z_az_{\ve{a}'}*z_bz_{\ve{b}'}
  =z_a(z_{\ve{a}'}*z_bz_{\ve{b}'})+z_b(z_az_{\ve{a}'}*z_{\ve{b}'})
   +z_{a+b}(z_{\ve{a}'}*z_{\ve{b}'})
 \]
 for $a,b\in\Z_{\ge1}$ and $\ve{a}',\ve{b}'\in I$.
 Consider the maps
 \begin{align*}
  S_{m,n+1}^*\to S_{m+1,n+1}^*&;(l,\sigma,\tau)\mapsto(l+1,\tilde{\sigma},\tau+1),\\
  S_{m+1,n}^*\to S_{m+1,n+1}^*&;(l,\sigma,\tau)\mapsto(l+1,\sigma+1,\tilde{\tau}),\\
  S_{m,n}^*\to S_{m+1,n+1}^*&;(l,\sigma,\tau)\mapsto(l+1,\tilde{\sigma},\tilde{\tau}),
 \end{align*}
 where $\tilde{\sigma}$, $\sigma+1$, $\tilde{\tau}$, $\tau+1$ are given by
 \begin{align*}
  \tilde{\sigma}&=\begin{pmatrix}1&2&\cdots&m+1\\1&\sigma(1)+1&\cdots&\sigma(m)+1\end{pmatrix},&
  \sigma+1&=\begin{pmatrix}1&\cdots&m+1\\\sigma(1)+1&\cdots&\sigma(m+1)+1\end{pmatrix},\\
  \tilde{\tau}&=\begin{pmatrix}1&2&\cdots&n+1\\1&\tau(1)+1&\cdots&\tau(n)+1\end{pmatrix},&
  \tau+1&=\begin{pmatrix}1&\cdots&n+1\\\tau(1)+1&\cdots&\tau(n+1)+1\end{pmatrix}.
 \end{align*}
 They are injective and their images are disjoint sets with union $S_{m+1,n+1}^*$.
 This completes the proof.
\end{proof}

\begin{prop}
 The map $\sha$ defined above is the same as the one defined by Muneta~\cite{Mun2}.
 It follows that $\sha$ is an associative and commutative product.
\end{prop}

\begin{proof}
 It suffices to prove that
 \[
  z_az_{\ve{a}'}\sha z_bz_{\ve{b}'}
  =z_a(z_{\ve{a}'}\sha z_bz_{\ve{b}'})+z_b(z_az_{\ve{a}'}\sha z_{\ve{b}'})
 \]
 for $a,b\in\Z_{\ge1}$ and $\ve{a}',\ve{b}'\in I$.
 Consider the maps
 \begin{align*}
  S_{m,n+1}^{\sha}\to S_{m+1,n+1}^{\sha}&;(l,\sigma,\tau)\mapsto(l+1,\tilde{\sigma},\tau+1),\\
  S_{m+1,n}^{\sha}\to S_{m+1,n+1}^{\sha}&;(l,\sigma,\tau)\mapsto(l+1,\sigma+1,\tilde{\tau}),
 \end{align*}
 where $\tilde{\sigma}$, $\sigma+1$, $\tilde{\tau}$, $\tau+1$ are given in the preceding proof.
 They are injective and their images are disjoint sets with union $S_{m+1,n+1}^{\sha}$.
 This completes the proof.
\end{proof}

The product $\sha$ allows us to write our main theorem (Theorem~\ref{thm:main})
in the following simple form:
\begin{thm}[Main Theorem]\label{thm:main_with_sha}
 For all $m,n\in\Z_{\ge0}$, we have
 \[
  \bar{Z}\bigl(z_2^m\sha(z_3z_1)^n\bigr)\in\Q\pi^{2m+4n}.
 \]
\end{thm}

\subsection{Reduction of our main theorem to an identity in $\H^1$}
The aim of this subsection is to reduce our main theorem to the following identity in $\H^1$:
\begin{thm}\label{mthm}
 For all $a,b,c\in\Z_{\ge1}$ and $m,n\in\Z_{\ge0}$, we have
 \begin{align*}
  &\sum_{\substack{i+p+2q=m\\j+u_1+\cdots+u_p\phantom{=n}\\\phantom{j}+v_1+\cdots+v_q=n}}
  (-2)^pd\bigl(z_{c}^i\sha(z_{a}z_{b})^j\bigr)*
  (z_{(a+b)u_1+c}\cdots z_{(a+b)u_p+c}\sha
   z_{(a+b)v_1+2c}\cdots z_{(a+b)v_q+2c})\\
  &\qquad\qquad=(-1)^m\sum_{j+k=n}\bigl(z_c^m\sha(z_az_b)^j\bigr)*d(z_{a+b}^k).
 \end{align*}
\end{thm}

Assuming Theorem~\ref{mthm} hereafter within this subsection,
we will give a proof of our main theorem.

\begin{defn}
 A \emph{partition} of a set $X$ is a family of pairwise disjoint nonempty subsets of $X$ with union $X$.
 For $n\in\Z_{\ge1}$, the collection of all partitions of $[n]$ is denoted by $\Pi_n$.
\end{defn}

\begin{ex}
 For $n=3$, we have
 \[
  \Pi_3=\Bigl\{\bigl\{\{1\},\{2\},\{3\}\bigr\},\bigl\{\{1,2\},\{3\}\bigr\},
   \bigl\{\{1,3\},\{2\}\bigr\},\bigl\{\{2,3\},\{1\}\bigr\},\bigl\{\{1,2,3\}\bigr\}\Bigr\}.
 \]
\end{ex}

\begin{lem}\label{lem:harm_sha}
 For all $a_1,\ldots,a_n\in\Z_{\ge1}$, we have
 \[
  z_{a_1}*\cdots*z_{a_n}
  =\sum_{\P\in\Pi_n}\bigsha_{A\in\P}z_{\sum_{i\in A}a_i}.
 \]
\end{lem}

\begin{proof}
 Easy. Also see \cite{Hof}. 
\end{proof}

\begin{ex}
 For $n=3$, we have
 \[
  z_{a_1}*z_{a_2}*z_{a_3}
  =z_{a_1}\sha z_{a_2}\sha z_{a_3}
   +z_{a_1+a_2}\sha z_{a_3}+z_{a_1+a_3}\sha z_{a_2}+z_{a_2+a_3}\sha z_{a_1}+z_{a_1+a_2+a_3}.
 \]
\end{ex}

\begin{lem}\label{lem:Z_shas}
 If $a_1,\ldots,a_n\in\Z_{\ge1}$ are all even, then
 \[
  Z(z_{a_1}\sha\cdots\sha z_{a_n})\in\Q\pi^{a_1+\cdots+a_n}.
 \]
\end{lem}

\begin{proof}
 We proceed by induction on $n$.
 For $n=1$, the assertion is obvious because
 \[
  Z(z_{a_1})=\zeta(a_1)\in\Q\pi^{a_1}.
 \]
 Suppose that the lemma holds for $1,\ldots,n-1$.
 Applying $Z$ to Lemma~\ref{lem:harm_sha} gives
 \begin{align*}
  \zeta(a_1)\cdots\zeta(a_n)
  &=Z\Biggl(\sum_{\P\in\Pi_n}\bigsha_{A\in\P}z_{\sum_{i\in A}a_i}\Biggr)\\
  &=Z(z_{a_1}\sha\cdots\sha z_{a_n})
   +\sum_{\P\in\Pi_n\setminus\{\{\{1\},\ldots,\{n\}\}\}}Z\Biggl(\bigsha_{A\in\P}z_{\sum_{i\in A}a_i}\Biggr).
 \end{align*}
 If $\P\in\Pi_n\setminus\bigl\{\bigl\{\{1\},\ldots,\{n\}\bigr\}\bigr\}$, then
 $\P$ has cardinality at most $n-1$, and so the inductive hypothesis shows that
 \[
  Z\Biggl(\bigsha_{A\in\P}z_{\sum_{i\in A}a_i}\Biggr)
  \in\Q\pi^{\sum_{A\in\P}\sum_{i\in A}a_i}
  =\Q\pi^{a_1+\cdots+a_n}.
 \]
 It follows that
 \[
  Z(z_{a_1}\sha\cdots\sha z_{a_n})
  =\zeta(a_1)\cdots\zeta(a_n)
  -\sum_{\P\in\Pi_n\setminus\{\{\{1\},\ldots,\{n\}\}\}}Z\Biggl(\bigsha_{A\in\P}z_{\sum_{i\in A}a_i}\Biggr)
  \in\Q\pi^{a_1+\cdots+a_n}.
 \]
\end{proof}

\begin{lem}\label{lem:sum_sha}
 For all $p,q,k,l\in\Z_{\ge0}$, we have
 \[
  \sum_{\substack{u_1+\cdots+u_p=k\\v_1+\cdots+v_q=l}}
  Z(z_{4u_1+2}\cdots z_{4u_p+2}\sha z_{4v_1+4}\cdots z_{4v_q+4})\in\Q\pi^{2p+4q+4k+4l}.
 \]
\end{lem}

\begin{proof}
 We first observe that
 \begin{align*}
  \sum_{u_1+\cdots+u_p=k}z_{4u_1+2}\sha\cdots\sha z_{4u_p+2}
  &=\sum_{\substack{u_1+\cdots+u_p=k\\\sigma\in\sym_p}}z_{4u_{\sigma(1)}+2}\cdots z_{4u_{\sigma(p)}+2}\\
  &=\sum_{\sigma\in\sym_p}\sum_{u_{\sigma(1)}+\cdots+u_{\sigma(p)}=k}
    z_{4u_{\sigma(1)}+2}\cdots z_{4u_{\sigma(p)}+2}\\
  &=\sum_{\sigma\in\sym_p}\sum_{u_1+\cdots+u_p=k}z_{4u_1+2}\cdots z_{4u_p+2}\\
  &=p!\sum_{u_1+\cdots+u_p=k}z_{4u_1+2}\cdots z_{4u_p+2}.
 \end{align*}
 This and a similar equation give
 \begin{align*}
  &\sum_{\substack{u_1+\cdots+u_p=k\\v_1+\cdots+v_q=l}}
   z_{4u_1+2}\cdots z_{4u_p+2}\sha z_{4v_1+4}\cdots z_{4v_q+4}\\
  &\qquad=\Biggl(\sum_{u_1+\cdots+u_p=k}z_{4u_1+2}\cdots z_{4u_p+2}\Biggr)\sha
   \Biggl(\sum_{v_1+\cdots+v_q=l}z_{4v_1+4}\cdots z_{4v_q+4}\Biggr)\\
  &\qquad=\frac{1}{p!}\Biggl(\sum_{u_1+\cdots+u_p=k}z_{4u_1+2}\sha\cdots\sha z_{4u_p+2}\Biggr)\sha
   \frac{1}{q!}\Biggl(\sum_{v_1+\cdots+v_q=l}z_{4v_1+4}\sha\cdots\sha z_{4v_q+4}\Biggr)\\
  &\qquad=\frac{1}{p!q!}\sum_{\substack{u_1+\cdots+u_p=k\\v_1+\cdots+v_q=l}}
   z_{4u_1+2}\sha\cdots\sha z_{4u_p+2}\sha z_{4v_1+4}\sha\cdots\sha z_{4v_q+4}.
 \end{align*}
 Hence the result follows from Lemma~\ref{lem:Z_shas}.
\end{proof}

\begin{lem}\label{lem:RHS}
 For all $m,n\in\Z_{\ge0}$, we have
 \[
  \sum_{j+k=n}Z\bigl(z_2^m\sha(z_3z_1)^j\bigr)\bar{Z}(z_4^k)\in\Q\pi^{2m+4n}.
 \]
\end{lem}

\begin{proof}
 It suffices to show that
 \[
  Z\bigl(z_2^m\sha(z_3z_1)^j\bigr)\bar{Z}(z_4^k)\in\Q\pi^{2m+4j+4k}
 \]
 for all $j,k\in\Z_{\ge0}$.
 Observe that
 \[
  Z\bigl(z_2^m\sha(z_3z_1)^j\bigr)
  =\sum_{\substack{\sum_{i=0}^{2j}m_i=m\\m_0,\ldots,m_{2j}\ge0}}
    \zeta(\{2\}^{m_0},3,\{2\}^{m_1},1,\ldots,3,\{2\}^{m_{2j-1}},1,\{2\}^{m_{2j}})
  \in\Q\pi^{2m+4j}
 \]
 by Bowman-Bradley~\cite{BowBra} and that
 \[
  \bar{Z}(z_4^k)=\zeta^{\star}(\{4\}^k)\in\Q\pi^{4k}
 \]
 by Muneta~\cite{Mun}.
 These observations complete the proof.
\end{proof}

Now we give a proof of our main theorem (Theorem~\ref{thm:main_with_sha}),
assuming Theorem~\ref{mthm}:

\begin{proof}[Proof of Theorem~\ref{thm:main_with_sha}]
 We proceed by induction on $m+n$
 and note that the theorem is obvious if $m+n=0$, in which case
 the only possibility is $m=n=0$.

 Apply $Z$ to Theorem~\ref{mthm} and substitute $(a,b,c)=(3,1,2)$ to get
 \begin{align*}
  &\sum_{\substack{i+p+2q=m\\j+u_1+\cdots+u_p\phantom{=n}\\\phantom{j}+v_1+\cdots+v_q=n}}
  (-2)^p\bar{Z}\bigl(z_2^i\sha(z_3z_1)^j\bigr)
  Z(z_{4u_1+2}\cdots z_{4u_p+2}\sha z_{4v_1+4}\cdots z_{4v_q+4})\\
  &\qquad\qquad=(-1)^m\sum_{j+k=n}Z\bigl(z_2^m\sha(z_3z_1)^j\bigr)\bar{Z}(z_4^k),
 \end{align*}
 whose right-hand side belongs to $\Q\pi^{2m+4n}$ by Lemma~\ref{lem:RHS}.
 The left-hand side is
 \begin{align*}
  &\sum_{i=0}^{m}\sum_{j=0}^{n}\bar{Z}\bigl(z_2^i\sha(z_3z_1)^j\bigr)
   \sum_{\substack{p+2q=m-i\\k+l=n-j}}(-2)^p
   \sum_{\substack{u_1+\cdots+u_p=k\\v_1+\cdots+v_q=l}}
   Z(z_{4u_1+2}\cdots z_{4u_p+2}\sha z_{4v_1+4}\cdots z_{4v_q+4})\\
  &\qquad\in\bar{Z}\bigl(z_2^m\sha(z_3z_1)^n\bigr)+\Q\pi^{2m+4n}
 \end{align*}
 by the inductive hypothesis and Lemma~\ref{lem:sum_sha}
 because if $p+2q=m-i$ and $k+l=n-j$, then $(2i+4j)+(2p+4q+4k+4l)=2m+4n$.
 This completes the proof.
\end{proof}

\subsection{Reduction to an identity in a larger algebra}
In order to prove Theorem~\ref{mthm}, we find it convenient to consider an algebra larger than $\H^1$.
Let $\M$ denote the sub-semigroup $\Z_{\ge0}^3\setminus\{(0,0,0)\}$ of $\Z^3$, and
$\A$ the non-commutative polynomial algebra $\Q\langle x_{\alpha}\mid\alpha\in\M\rangle$.

Write $\I$ for the set of all finite sequences of elements of $\M$, including the empty sequence $\emptyset$.
For $\ve{\alpha}=(\alpha_1,\ldots,\alpha_l)\in\I$,
put $x_{\ve{\alpha}}=x_{\alpha_1}\cdots x_{\alpha_n}\in\A$, where $z_{\emptyset}=1$;
then $\{x_{\ve{\alpha}}\mid\ve{\alpha}\in\I\}$ is a $\Q$-vector space basis for $\A$.

We may define a $\Q$-linear transformation $d$ on $\A$
and $\Q$-bilinear maps $*,\sha\colon\A\times\A\to\A$ analogously as we did on $\H^1$.
Then for any $a,b,c\in\Z_{\ge1}$,
the algebra homomorphism $\A\to\H^1$ defined by
$x_{\alpha}\mapsto z_{ap+bq+cr}$ for $\alpha=(p,q,r)\in\M$ commutes with $d$, $*$ and $\sha$.
Therefore, in order to prove Theorem~\ref{mthm},
it suffices to show the following identity in $\A$:
\begin{thm}\label{thm:in A}
 For all $m,n\in\Z_{\ge0}$, we have
 \begin{align*}
  &\sum_{\substack{i+p+2q=m\\j+u_1+\cdots+u_p\phantom{=n}\\\phantom{j}+v_1+\cdots+v_q=n}}
  (-2)^pd\bigl(x_{e_3}^i\sha(x_{e_1}x_{e_2})^j\bigr)*
  (x_{(u_1,u_1,1)}\cdots x_{(u_p,u_p,1)}\sha x_{(v_1,v_1,2)}\cdots x_{(v_q,v_q,2)})\\
  &\qquad\qquad=(-1)^m\sum_{j+k=n}\bigl(x_{e_3}^m\sha(x_{e_1}x_{e_2})^j\bigr)*d(x_{(1,1,0)}^k),
 \end{align*}
 where $e_1=(1,0,0)$, $e_2=(0,1,0)$, $e_3=(0,0,1)$.
\end{thm}

\subsection{Reduction to an identity in an algebra of formal power series}
\begin{defn}
 For $\ve{\alpha}=(\alpha_1,\ldots,\alpha_l)\in\I$,
 its \emph{weight} $\wt\ve{\alpha}$ and \emph{depth} $\dep\ve{\alpha}$ are defined by
 $\wt\ve{\alpha}=\alpha_1+\cdots+\alpha_l\in\Z_{\ge0}^3$ and $\dep\ve{\alpha}=l\in\Z_{\ge0}$,
 where the empty sequence $\emptyset$ is understood to have weight $0=(0,0,0)$ and depth $0$.
 For $\alpha\in\Z_{\ge0}^3$,
 the set of all $\ve{\alpha}\in\I$ of weight $\alpha$
 is denoted by $\I_{\wt=\alpha}$, and
 the $\Q$-vector subspace of $\A$ generated by
 $\{x_{\ve{\alpha}}\mid\ve{\alpha}\in\I_{\wt=\alpha}\}$ is denoted by $\A_{\wt=\alpha}$.
\end{defn}

\begin{prop}\label{prop:wt_dwsha}
 If $w\in\A_{\wt=\alpha}$ and $w'\in\A_{\wt=\alpha'}$,
 then $d(w)\in\A_{\wt=\alpha}$ and $w*w',w\sha w'\in\A_{\wt=\alpha+\alpha'}$.
\end{prop}

\begin{proof}
 We only prove that $d(w)\in\A_{\wt=\alpha}$; the other assertions can be shown in a similar manner.
 We may assume that $w=x_{\ve{\alpha}}$, where $\wt\ve{\alpha}=\alpha$.
 Put $k=\dep\ve{\alpha}$ and write $\ve{\alpha}=(\alpha_1,\ldots,\alpha_k)$.
 Let $(l,\sigma)\in S_k^d$ and define $\ve{\beta}=(\beta_1,\ldots,\beta_l)$ by
 $\beta_t=\sum_{s\in\sigma^{-1}(t)}\alpha_s$.
 Then we have
 \[
  \wt\ve{\beta}=\sum_{t=1}^{l}\beta_t=\sum_{t=1}^{l}\sum_{s\in\sigma^{-1}(t)}\alpha_s
  =\sum_{s=1}^{k}\alpha_s=\wt\ve{\alpha}=\alpha.
 \]
 It follows that $x_{\ve{\beta}}\in\A_{\wt=\alpha}$, which completes the proof.
\end{proof}

The proposition above implies that both sides in Theorem~\ref{thm:in A}
have weight $(n,n,m)$.
Therefore it suffices to show that
\begin{align*}
 &\sum_{m,n}\sum_{\substack{i+p+2q=m\\j+u_1+\cdots+u_p\phantom{=n}\\\phantom{j}+v_1+\cdots+v_q=n}}
 (-2)^pd\bigl(x_{e_3}^i\sha(x_{e_1}x_{e_2})^j\bigr)*
 (x_{(u_1,u_1,1)}\cdots x_{(u_p,u_p,1)}\sha x_{(v_1,v_1,2)}\cdots x_{(v_q,v_q,2)})\\
 &\qquad\qquad=\sum_{m,n}(-1)^m\sum_{j+k=n}\bigl(x_{e_3}^m\sha(x_{e_1}x_{e_2})^j\bigr)*d(x_{(1,1,0)}^k)
\end{align*}
in the algebra $\Q\langle\langle x_{\alpha}\mid\alpha\in\M\rangle\rangle$ of formal power series.
Observe that $d$, $*$ and $\sha$ are well defined in
$\Q\langle\langle x_{\alpha}\mid\alpha\in\M\rangle\rangle$.
It follows that the following theorem implies our main theorem:
\begin{thm}\label{thm:in_formal}
 We have
 \begin{align*}
  &\sum_{i,j,p,q,u_1,\ldots,u_p,v_1,\ldots,v_q}
  (-2)^pd\bigl(x_{e_3}^i\sha(x_{e_1}x_{e_2})^j\bigr)*
  (x_{(u_1,u_1,1)}\cdots x_{(u_p,u_p,1)}\sha x_{(v_1,v_1,2)}\cdots x_{(v_q,v_q,2)})\\
  &\qquad\qquad=\sum_{j,k,m}(-1)^m\bigl(x_{e_3}^m\sha(x_{e_1}x_{e_2})^j\bigr)*d(x_{(1,1,0)}^k)
 \end{align*}
 in $\Q\langle\langle x_{\alpha}\mid\alpha\in\M\rangle\rangle$.
\end{thm}

\section{Proof of Theorem~\ref{thm:in_formal}}
This section will be devoted to the proof of Theorem~\ref{thm:in_formal} by expanding both sides.

\subsection{Basic properties of the operations}
\begin{defn}
 We say that $\ve{\alpha}=(\alpha_1,\ldots,\alpha_l)\in\I$ is \emph{good}
 if, writing $\alpha_s=(a_s,b_s,c_s)$ for $s\in[l]$,
 we have $\sum_{s=1}^{s'}(a_s-b_s)\in\{0,1\}$ for all $s'\in[l]$
 and $\sum_{s=1}^{l}(a_s-b_s)=0$.
 The set of all good $\ve{\alpha}\in\I$ is denoted by $\I^{\G}$,
 and write
 $\A^{\G}=\{\sum_{\ve{\alpha}\in\I^{\G}}q_{\ve{\alpha}}x_{\ve{\alpha}}\mid q_{\ve{\alpha}}\in\Q\}$,
 which is a $\Q$-vector subspace of $\Q\langle\langle x_{\alpha}\mid\alpha\in\M\rangle\rangle$.

 We say that $\alpha\in\M$ is \emph{great} if, writing $\alpha=(a,b,c)$, we have $\lvert a-b\rvert\le1$.
 The set of all great $\alpha\in\M$ is denoted by $\M^{\G}$.
\end{defn}

\begin{lem}
 If $\ve{\alpha}=(\alpha_1,\ldots,\alpha_l)\in\I$ is good, then
 $\alpha_s$ is great for all $s\in[l]$.
\end{lem}

\begin{proof}
 Obvious.
\end{proof}

\begin{prop}\label{prop:good_dwsha}
 If $w,w'\in\A^{\G}$, then $d(w),w*w',w\sha w'\in\A^{\G}$.
\end{prop}

\begin{proof}
 We only prove that $d(w)\in\A^{\G}$; the other assertions can be shown in a similar manner.
 We may assume that $w=x_{\ve{\alpha}}$, where $\ve{\alpha}\in\I^{\G}$.
 Put $k=\dep\ve{\alpha}$ and write $\ve{\alpha}=(\alpha_1,\ldots,\alpha_k)$ and
 $\alpha_s=(a_s,b_s,c_s)$ for $s\in[k]$.
 Let $(l,\sigma)\in S_k^d$ and define $\ve{\beta}=(\beta_1,\ldots,\beta_l)$ by
 $\beta_t=\sum_{s\in\sigma^{-1}(t)}\alpha_s$.
 If we write $\beta_t=(\tilde{a}_t,\tilde{b}_t,\tilde{c}_t)$ for $t\in[l]$, then
 \[
  \sum_{t=1}^{t'}(\tilde{a}_t-\tilde{b}_t)
  =\sum_{t=1}^{t'}\sum_{s\in\sigma^{-1}(t)}(a_s-b_s)
  =\sum_{s=1}^{\max\sigma^{-1}(t')}(a_s-b_s)
  \in\{0,1\}
 \]
 for all $t'\in[l]$ because $\ve{\alpha}$ is good.
 It follows that $x_{\ve{\beta}}\in\A^{\G}$, which completes the proof.
\end{proof}

\begin{cor}
 Both sides in Theorem~\ref{thm:in_formal} belong to $\A^{\G}$.
\end{cor}

\begin{proof}
 The corollary follows from Proposition~\ref{prop:good_dwsha} and the observation that
 $x_{e_3}^i$, $(x_{e_1}x_{e_2})^j$, $x_{(1,1,0)}^k$,
 $x_{(u_1,u_1,1)}\cdots x_{(u_p,u_p,1)}$ and
 $x_{(v_1,v_1,2)}\cdots x_{(v_q,v_q,2)}$ all belong to $\A^{\G}$.
\end{proof}

\begin{defn}
 For $\M'\subset\M$,
 write $\I_{\M'}$ for the set of all $\ve{\alpha}\in\I$ whose components are all in $\M'$,
 and put $\I_{\M'}^{\G}=\I_{\M'}\cap\I^{\G}$.
\end{defn}

\begin{lem}\label{lem:good_*}
 Let $\M',\M''\subset\M$ and suppose that $\M''\subset\{(a,a,b)\mid a,b\in\Z_{\ge0}\}$.
 Given $A_{\beta},B_{\gamma}\in\Q$ for each $\beta\in\M'$ and $\gamma\in\M''$,
 write
 \[
  A_{\ve{\beta}}=\prod_{s=1}^{\dep\ve{\beta}}A_{\beta_s},\qquad
  B_{\ve{\gamma}}=\prod_{t=1}^{\dep\ve{\gamma}}B_{\gamma_t}
 \]
 for $\ve{\beta}=(\beta_1,\ldots,\beta_{\dep\ve{\beta}})\in\I_{\M'}$ and
 $\ve{\gamma}=(\gamma_1,\ldots,\gamma_{\dep\ve{\gamma}})\in\I_{\M''}$.
 Then we have
 \[
  \sum_{\ve{\beta}\in\I_{\M'}^{\G}}\sum_{\ve{\gamma}\in\I_{\M''}}
  A_{\ve{\beta}}B_{\ve{\gamma}}(x_{\ve{\beta}}*x_{\ve{\gamma}})
  =\sum_{\ve{\alpha}=(\alpha_1,\ldots,\alpha_l)\in\I^{\G}}
   \Biggl(\prod_{s=1}^{l}\sum_{\substack{\beta\in\M'\cup\{0\}\\\gamma\in\M''\cup\{0\}\\\beta+\gamma=\alpha_s}}
   A_{\beta}B_{\gamma}\Biggr)x_{\ve{\alpha}},
 \]
 where we define $A_0=B_0=1$.
\end{lem}

\begin{proof}
 Since $\I_{\M''}\subset\I^{\G}$ by assumption,
 the left-hand side belongs to $\A^{\G}$ because of Proposition~\ref{prop:good_dwsha}.
 Therefore it suffices to show that for each $\ve{\alpha}=(\alpha_1,\ldots,\alpha_l)\in\I^{\G}$,
 the coefficient of $x_{\ve{\alpha}}$ in the left-hand side is equal to
 \[
  \prod_{s=1}^{l}\sum_{\substack{\beta\in\M'\cup\{0\}\\\gamma\in\M''\cup\{0\}\\\beta+\gamma=\alpha_s}}
  A_{\beta}B_{\gamma}
  =\sum_{\substack{\tilde{\beta}_1,\ldots,\tilde{\beta}_l\in\M'\cup\{0\}\\
    \tilde{\gamma}_1,\ldots,\tilde{\gamma}_l\in\M''\cup\{0\}\\
    \tilde{\beta}_s+\tilde{\gamma}_s=\alpha_s\,\text{for}\,s\in[l]}}
   A_{\tilde{\beta}_1}\cdots A_{\tilde{\beta}_l}B_{\tilde{\gamma}_1}\cdots B_{\tilde{\gamma}_l}.
 \]

 In the left-hand side, each $(\ve{\beta},\ve{\gamma},\sigma,\tau)$
 with $\ve{\beta}=(\beta_1,\ldots,\beta_{\dep\ve{\beta}})\in\I_{\M'}^{\G}$,
 $\ve{\gamma}=(\gamma_1,\ldots,\gamma_{\dep\ve{\gamma}})\in\I_{\M''}$,
 $(\dep\ve{\alpha},\sigma,\tau)\in S_{\dep\ve{\beta},\dep\ve{\gamma}}^*$
 satisfying $\alpha_s=\sum_{t\in\sigma^{-1}(s)}\beta_t+\sum_{u\in\tau^{-1}(s)}\gamma_u$
 for $s\in[l]$ yields the coefficient $A_{\ve{\beta}}B_{\ve{\gamma}}$.
 For each such $(\ve{\beta},\ve{\gamma},\sigma,\tau)$,
 define $\tilde{\beta}_s\in\M'\cup\{0\}$ for $s\in[l]$ by
 \[
  \tilde{\beta}_s=
  \begin{cases}
   \beta_t&\text{if $s\in\Im\sigma$ and $\sigma(t)=s$};\\
   0&\text{if $s\notin\Im\sigma$}.
  \end{cases}
 \]
 Define $\tilde{\gamma}_s\in\M''\cup\{0\}$ for $s\in[l]$ in a similar manner.
 Then we have $\tilde{\beta}_s+\tilde{\gamma}_s=\alpha_s$ for all $s\in[l]$,
 and
 \[
  A_{\ve{\beta}}B_{\ve{\gamma}}
  =A_{\tilde{\beta}_1}\cdots A_{\tilde{\beta}_l}B_{\tilde{\gamma}_1}\cdots B_{\tilde{\gamma}_l}.
 \]

 Conversely, let
 $\tilde{\beta}_1,\ldots,\tilde{\beta}_l\in\M'\cup\{0\}$ and
 $\tilde{\gamma}_1,\ldots,\tilde{\gamma}_l\in\M''\cup\{0\}$ be given so that
 $\tilde{\beta}_s+\tilde{\gamma}_s=\alpha_s$ for all $s\in[l]$.
 Let $\ve{\beta}$ be the sequence $(\tilde{\beta}_1,\ldots,\tilde{\beta}_l)$
 with all zero components removed,
 and define $\sigma\colon[\dep\ve{\beta}]\to[l]$ by setting $\sigma(t)=s$
 if the $t$-th component of $\ve{\beta}$ was originally $\tilde{\beta}_s$.
 Note that the goodness of $\ve{\alpha}$ implies that of $\ve{\beta}$
 thanks to the assumption on $\M''$.
 Defining $\ve{\gamma}$ and $\tau$ in a similar fashion,
 we may see that this gives the desired one-to-one correspondence.
\end{proof}

\subsection{Expansion of the right-hand side}
\begin{defn}
 For a finite subset $A$ of $\Z$,
 denote by $A_{\odd}$ and $A_{\even}$ the sets of
 all odd and even elements of $A$ respectively, and
 by $\lvert A\rvert$ the cardinality of $A$.
\end{defn}

\begin{defn}
 For $a\in\Z_{\ge0}$, set $\M_a=\{(b,b,a)\mid b\in\Z_{\ge0}\}\setminus\{(0,0,0)\}$.
\end{defn}

\begin{lem}\label{lem:RHS1}
 We have
 \[
  \sum_{k}d(x_{(1,1,0)}^k)=\sum_{\ve{\gamma}\in\I_{\M_0}}x_{\ve{\gamma}}.
 \]
\end{lem}

\begin{proof}
 Since
 \[
  \sum_{k}d(x_{(1,1,0)}^k)=\sum_{k}\sum_{(l,\sigma)\in S_k^d}x_{\ve{\gamma}},
 \]
 where $\ve{\gamma}=(\gamma_1,\ldots,\gamma_l)\in\I$ is given by
 \[
  \gamma_i
  =\bigl(\lvert\sigma^{-1}(i)\rvert,\lvert\sigma^{-1}(i)\rvert,0\bigr)\in\M_0,
 \]
 we only need to show that the map $(k,l,\sigma)\mapsto\ve{\gamma}$ is a bijection onto $\I_{\M_0}$.
 Let $\ve{\gamma}\in\I_{\M_0}$ be given.
 Set $l=\dep\ve{\gamma}$ and write $\ve{\gamma}=(\gamma_1,\ldots,\gamma_l)$.
 Put $k=(\wt\ve{\gamma})\cdot e_1$,
 and define $\sigma\colon[k]\to[l]$ by
 \[
  \sigma(s)=\min\{i\in[l]\mid s\le(\gamma_1+\cdots+\gamma_i)\cdot e_1\}.
 \]
 It is easily seen that this gives the inverse, which completes the proof.
\end{proof}

\begin{lem}\label{lem:RHS2}
 We have
 \[
  \sum_{j,m}(-1)^m\bigl(x_{e_3}^m\sha(x_{e_1}x_{e_2})^j\bigr)
  =\sum_{\ve{\beta}\in\I_{\{e_1,e_2,e_3\}}^{\G}}(-1)^{(\wt\ve{\beta})\cdot e_3}x_{\ve{\beta}}.
 \]
\end{lem}

\begin{proof}
 Since
 \begin{align*}
  \sum_{j,m}(-1)^m\bigl(x_{e_3}^m\sha(x_{e_1}x_{e_2})^j\bigr)
  &=\sum_{j,m}\sum_{(l,\sigma,\tau)\in S_{m,2j}^{\sha}}(-1)^mx_{\ve{\beta}},
 \end{align*}
 where $\ve{\beta}=(\beta_1,\ldots,\beta_l)\in\I$ is given by
 \[
  \beta_i=
  \begin{cases}
   e_3&\text{if $i\in\Im\sigma$};\\
   e_1&\text{if $i\in\tau([2j]_{\odd})$};\\
   e_2&\text{if $i\in\tau([2j]_{\even})$},
  \end{cases}
 \]
 we only need to show that the map $(j,m,l,\sigma,\tau)\mapsto\ve{\beta}$
 is a bijection onto $\I_{\{e_1,e_2,e_3\}}^{\G}$, because the coefficients are seen to agree by
 the observation that
 \[
  (-1)^m=(-1)^{\lvert\{i\in[l]\mid \beta_i=e_3\}\rvert}=(-1)^{(\wt\ve{\beta})\cdot e_3}.
 \]
 Let $\ve{\beta}\in\I_{\{e_1,e_2,e_3\}}^{\G}$ be given.
 Set $l=\dep\ve{\beta}$ and write $\ve{\beta}=(\beta_1,\ldots,\beta_l)$.
 Put $j=(\wt\ve{\beta})\cdot e_1$ and $m=(\wt\ve{\beta})\cdot e_3$,
 and define $\sigma\colon[m]\to[l]$ and $\tau\colon[2j]\to[l]$ by
 \begin{align*}
  \sigma(s)&=\text{$s$-th smallest element in $\{i\in[l]\mid\beta_i=e_3\}$},\\
  \tau(t)&=\text{$t$-th smallest element in $\{i\in[l]\mid\beta_i=e_1,e_2\}$}.
 \end{align*}
 It is easily seen that this gives the inverse, which completes the proof.
\end{proof}

\begin{defn}
 Define
 \begin{align*}
  \N&=\{(a,a,0)\mid a\in\Z_{\ge1}\}\cup\{(a+1,a,0)\mid a\in\Z_{\ge0}\}\\
    &\qquad\cup\{(a,a+1,0)\mid a\in\Z_{\ge0}\}\cup\{(a,a,1)\mid a\in\Z_{\ge0}\}\\
    &\subset\M^{\G}.
 \end{align*}
\end{defn}

\begin{lem}\label{lem:RHS_final}
 We have
 \[
  \sum_{j,k,m}(-1)^m\bigl(x_{e_3}^m\sha(x_{e_1}x_{e_2})^j\bigr)*d(x_{(1,1,0)}^k)
  =\sum_{\ve{\alpha}\in\I_{\N}^{\G}}(-1)^{(\wt\ve{\alpha})\cdot e_3}x_{\ve{\alpha}}.
 \]
\end{lem}

\begin{proof}
 Lemmas~\ref{lem:RHS1} and \ref{lem:RHS2} show that
 \begin{align*}
  &\sum_{j,k,m}(-1)^m\bigl(x_{e_3}^m\sha(x_{e_1}x_{e_2})^j\bigr)*d(x_{(1,1,0)}^k)\\
  &\qquad=
   \Biggl(\sum_{j,m}(-1)^m\bigl(x_{e_3}^m\sha(x_{e_1}x_{e_2})^j\bigr)\Biggr)*
   \Biggl(\sum_{k}d(x_{(1,1,0)}^k)\Biggr)\\
  &\qquad=
   \Biggl(\sum_{\ve{\beta}\in\I_{\{e_1,e_2,e_3\}}^{\G}}(-1)^{(\wt\ve{\beta})\cdot e_3}x_{\ve{\beta}}\Biggr)*
   \Biggl(\sum_{\ve{\gamma}\in\I_{\M_0}}x_{\ve{\gamma}}\Biggr)\\
  &\qquad=
   \sum_{\ve{\beta}\in\I_{\{e_1,e_2,e_3\}}^{\G}}\sum_{\ve{\gamma}\in\I_{\M_0}}
   (-1)^{(\wt\ve{\beta})\cdot e_3}(x_{\ve{\beta}}*x_{\ve{\gamma}}).
 \end{align*}

 Applying Lemma~\ref{lem:good_*} with $\M'=\{e_1,e_2,e_3\}$, $\M''=\M_0$,
 $A_{\beta}=(-1)^{\beta\cdot e_3}$ and $B_{\gamma}=1$ gives
 \begin{align*}
  &\sum_{\ve{\beta}\in\I_{\{e_1,e_2,e_3\}}^{\G}}\sum_{\ve{\gamma}\in\I_{\M_0}}
   (-1)^{(\wt\ve{\beta})\cdot e_3}(x_{\ve{\beta}}*x_{\ve{\gamma}})\\
  &\qquad=\sum_{\ve{\alpha}=(\alpha_1,\ldots,\alpha_l)\in\I^{\G}}
   \Biggl(\prod_{s=1}^{l}\sum_{\substack{\beta\in\{e_1,e_2,e_3,0\}\\\gamma\in\M_0\cup\{0\}\\\beta+\gamma=\alpha_s}}
   (-1)^{\beta\cdot e_3}\Biggr)x_{\ve{\alpha}}.
 \end{align*}

 For $\alpha=(a,b,c)\in\M^{\G}$, we have
 \begin{align*}
  \sum_{\substack{\beta\in\{e_1,e_2,e_3,0\}\\\gamma\in\M_0\cup\{0\}\\\beta+\gamma=\alpha}}
  (-1)^{\beta\cdot e_3}&=
  \begin{cases}
   1&\text{if $c=0$};\\
   -1&\text{if $a=b$ and $c=1$};\\
   0&\text{otherwise}
  \end{cases}\\
  &=
  \begin{cases}
   (-1)^{\alpha\cdot e_3}&\text{if $\alpha\in\N$};\\
   0&\text{if $\alpha\notin\N$}.
  \end{cases}
 \end{align*}

 It follows that
 \[
  \sum_{j,k,m}(-1)^m\bigl(x_{e_3}^m\sha(x_{e_1}x_{e_2})^j\bigr)*d(x_{(1,1,0)}^k)\\
  =\sum_{\ve{\alpha}\in\I_{\N}^{\G}}(-1)^{(\wt\ve{\alpha})\cdot e_3}x_{\ve{\alpha}}.
 \]
\end{proof}

\subsection{Expansion of the left-hand side}
\begin{defn}
 For $\ve{\gamma}=(\gamma_1,\ldots,\gamma_l)\in\I_{\M_1\cup\M_2}$, define
 \[
  \chara\ve{\gamma}=\lvert\{s\in[l]\mid\gamma_s\cdot e_3=1\}\rvert.
 \]
\end{defn}

\begin{lem}\label{lem:LHS1}
 We have
 \[
  \sum_{p,q,u_1,\ldots,u_p,v_1,\ldots,v_q}
  (-2)^p(x_{(u_1,u_1,1)}\cdots x_{(u_p,u_p,1)}\sha x_{(v_1,v_1,2)}\cdots x_{(v_q,v_q,2)})
  =\sum_{\ve{\gamma}\in\I_{\M_1\cup\M_2}}(-2)^{\chara\ve{\gamma}}x_{\ve{\gamma}}.
 \]
\end{lem}

\begin{proof}
 We have
 \begin{align*}
  &\sum_{p,q,u_1,\ldots,u_p,v_1,\ldots,v_q}
   (-2)^p(x_{(u_1,u_1,1)}\cdots x_{(u_p,u_p,1)}\sha x_{(v_1,v_1,2)}\cdots x_{(v_q,v_q,2)})\\
  &\qquad\qquad=\Biggl(\sum_{p,u_1,\ldots,u_p}(-2)^px_{(u_1,u_1,1)}\cdots x_{(u_p,u_p,1)}\Biggr)
   \sha\Biggl(\sum_{q,v_1,\ldots,v_q}x_{(v_1,v_1,2)}\cdots x_{(v_q,v_q,2)}\Biggr)\\
  &\qquad\qquad=\Biggl(\sum_{\ve{\gamma}_1\in\I_{\M_1}}(-2)^{\dep\ve{\gamma}_1}x_{\ve{\gamma}_1}\Biggr)
   \sha\Biggl(\sum_{\ve{\gamma}_2\in\I_{\M_2}}x_{\ve{\gamma}_2}\Biggr)\\
  &\qquad\qquad=\sum_{\ve{\gamma}_1\in\I_{\M_1}}\sum_{\ve{\gamma}_2\in\I_{\M_2}}
   (-2)^{\dep\ve{\gamma}_1}(x_{\ve{\gamma}_1}\sha x_{\ve{\gamma}_2})\\
  &\qquad\qquad=\sum_{\ve{\gamma}\in\I_{\M_1\cup\M_2}}
   (-2)^{\chara\ve{\gamma}}x_{\ve{\gamma}},
 \end{align*}
 where the last equality follows from the observation that
 if $\ve{\gamma}_1=(\gamma_{1,1},\ldots,\gamma_{1,\dep\ve{\gamma}_1})\in\I_{\M_1}$,
 $\ve{\gamma}_2=(\gamma_{2,1},\ldots,\gamma_{2,\dep\ve{\gamma}_2})\in\I_{\M_2}$
 and $(l,\sigma_1,\sigma_2)\in S_{\dep\ve{\gamma}_1,\dep\ve{\gamma}_2}^{\sha}$,
 and if we define $\ve{\gamma}=(\gamma_1,\ldots,\gamma_l)\in\I_{\M_1\cup\M_2}$
 by
 \[
  \gamma_s=\sum_{t\in\sigma_1^{-1}(s)}\gamma_{1,t}+\sum_{t\in\sigma_2^{-1}(s)}\gamma_{2,t},
 \]
 then
 \[
  \gamma_s\cdot e_3=
  \begin{cases}
   \gamma_{1,t}\cdot e_3=1&\text{if $s\in\Im\sigma_1$ and $s=\sigma_1(t)$};\\
   \gamma_{2,t}\cdot e_3=2&\text{if $s\in\Im\sigma_2$ and $s=\sigma_2(t)$},
  \end{cases}
 \]
 and so
 \[
  \chara\ve{\gamma}=\lvert\{s\in[l]\mid\gamma_s\cdot e_3=1\}\rvert
  =\lvert\Im\sigma_1\rvert=\dep\ve{\gamma}_1.
 \]
\end{proof}

\begin{defn}
 For $\alpha=(a,b,c)\in\M$, write
 \[
  \langle\alpha\rangle=\binom{a+b+c}{c}=\frac{(a+b+c)!}{(a+b)!c!}\in\Z_{\ge1}.
 \]
 For $\ve{\alpha}=(\alpha_1,\ldots,\alpha_l)\in\I$,
 write $\langle\ve{\alpha}\rangle=\prod_{i=1}^{l}\langle\alpha_i\rangle\in\Z_{\ge1}$,
 where $\langle\emptyset\rangle=1$.
\end{defn}

\begin{lem}\label{lem:LHS2}
 We have
 \[
  \sum_{i,j}d\bigl(x_{e_3}^i\sha(x_{e_1}x_{e_2})^j\bigr)
  =\sum_{\ve{\beta}\in\I^{\G}}\langle\ve{\beta}\rangle x_{\ve{\beta}}.
 \]
\end{lem}

\begin{proof}
 Since
 \[
  \sum_{i,j}(x_{e_3}^i\sha(x_{e_1}x_{e_2})^j\bigr)
  =\sum_{\ve{\alpha}\in\I_{\{e_1,e_2,e_3\}}^{\G}}x_{\ve{\alpha}}
 \]
 by a reasoning similar to the one used in the proof of Lemma~\ref{lem:RHS2}, we have
 \[
  \sum_{i,j}d\bigl(x_{e_3}^i\sha(x_{e_1}x_{e_2})^j\bigr)
  =\sum_{\ve{\alpha}\in\I_{\{e_1,e_2,e_3\}}^{\G}}d(x_{\ve{\alpha}})
  =\sum_{\ve{\alpha}\in\I_{\{e_1,e_2,e_3\}}^{\G}}\sum_{(l,\sigma)\in S_{\dep\ve{\alpha}}^d}x_{\ve{\beta}},
 \]
 where $\ve{\beta}=(\beta_1,\ldots,\beta_l)\in\I^{\G}$ is given by
 \[
  \beta_s=\sum_{t\in\sigma^{-1}(s)}\alpha_t
 \]
 if we write $\ve{\alpha}=(\alpha_1,\ldots,\alpha_{\dep\ve{\alpha}})$.
 It suffices to prove that for each $\ve{\beta}\in\I^{\G}$,
 there are $\langle\ve{\beta}\rangle$ triples $(\ve{\alpha},l,\sigma)$
 that yield $\ve{\beta}$.

 Let $\ve{\beta}\in\I^{\G}$ be given.
 Set $l=\dep\ve{\beta}$ and write $\ve{\beta}=(\beta_1,\ldots,\beta_l)$ and
 $\beta_s=(a_s,b_s,c_s)$ for $s\in[l]$.
 The depth of $\ve{\alpha}$ is uniquely determined because
 \[
  \dep\ve{\alpha}=\wt\ve{\alpha}\cdot(1,1,1)
  =\wt\ve{\beta}\cdot(1,1,1).
 \]
 The map $\sigma$ is also uniquely determined because
 for every $s\in[l]$, we have
 \[
  \lvert\sigma^{-1}(s)\rvert
  =\sum_{t\in\sigma^{-1}(s)}\alpha_t\cdot(1,1,1)
  =\beta_s\cdot(1,1,1)
  =a_s+b_s+c_s.
 \]
 Therefore choosing an appropriate $\ve{\alpha}$
 is equivalent to choosing, for each $s\in[l]$, from $a_s+b_s+c_s$ components
 $c_s$ components to be occupied by $e_3$,
 because then the goodness of $\ve{\alpha}$ determines how to allocate
 the remaining $a_s+b_s$ components to $e_1$ and $e_2$.
 It follows that the number of suitable $\ve{\alpha}$
 is
 \[
  \prod_{s=1}^{l}\binom{a_s+b_s+c_s}{c_s}=\langle\ve{\beta}\rangle.
 \]
\end{proof}

\begin{lem}\label{lem:LHS_final}
 We have
 \begin{align*}
  &\sum_{i,j,p,q,u_1,\ldots,u_p,v_1,\ldots,v_q}
   (-2)^pd\bigl(x_{e_3}^i\sha(x_{e_1}x_{e_2})^j\bigr)*
   (x_{(u_1,u_1,1)}\cdots x_{(u_p,u_p,1)}\sha x_{(v_1,v_1,2)}\cdots x_{(v_q,v_q,2)})\\
  &\qquad\qquad=\sum_{\ve{\alpha}\in\I_{\N}^{\G}}(-1)^{(\wt\ve{\alpha})\cdot e_3}x_{\ve{\alpha}}.
 \end{align*}
\end{lem}

\begin{proof}
 Lemmas~\ref{lem:LHS1} and \ref{lem:LHS2} show that
 \begin{align*}
  &\sum_{i,j,p,q,u_1,\ldots,u_p,v_1,\ldots,v_q}
   (-2)^pd\bigl(x_{e_3}^i\sha(x_{e_1}x_{e_2})^j\bigr)*
   (x_{(u_1,u_1,1)}\cdots x_{(u_p,u_p,1)}\sha x_{(v_1,v_1,2)}\cdots x_{(v_q,v_q,2)})\\
  &\qquad\qquad=
   \Biggl(\sum_{i,j}d\bigl(x_{e_3}^i\sha(x_{e_1}x_{e_2})^j\bigr)\Biggr)\\
  &\qquad\qquad\qquad\qquad*\Biggl(\sum_{p,q,u_1,\ldots,u_p,v_1,\ldots,v_q}
    (-2)^p(x_{(u_1,u_1,1)}\cdots x_{(u_p,u_p,1)}\sha x_{(v_1,v_1,2)}\cdots x_{(v_q,v_q,2)})
    \Biggr)\\
  &\qquad\qquad=
   \Biggl(\sum_{\ve{\beta}\in\I^{\G}}\langle\ve{\beta}\rangle x_{\ve{\beta}}\Biggr)
   *\Biggl(\sum_{\ve{\gamma}\in\I_{\M_1\cup\M_2}}(-2)^{\chara\ve{\gamma}}x_{\ve{\gamma}}\Biggr)\\
  &\qquad\qquad=\sum_{\ve{\beta}\in\I^{\G}}\sum_{\ve{\gamma}\in\I_{\M_1\cup\M_2}}
   \langle\ve{\beta}\rangle(-2)^{\chara\ve{\gamma}}(x_{\ve{\beta}}*x_{\ve{\gamma}}).
 \end{align*}

 Applying Lemma~\ref{lem:good_*} with $\M'=\M$, $\M''=\M_1\cup\M_2$,
 $A_{\beta}=\langle\beta\rangle$ and $B_{\gamma}=(-2)^{\chi_{\{1\}}(\gamma\cdot e_3)}$ gives
 \begin{align*}
  &\sum_{\ve{\beta}\in\I^{\G}}\sum_{\ve{\gamma}\in\I_{\M_1\cup\M_2}}
  \langle\ve{\beta}\rangle(-2)^{\chara\ve{\gamma}}(x_{\ve{\beta}}*x_{\ve{\gamma}})\\
  &\qquad=
   \sum_{\ve{\alpha}=(\alpha_1,\ldots,\alpha_l)\in\I^{\G}}
   \Biggl(\prod_{s=1}^{l}
    \sum_{\substack{\beta\in\M\cup\{0\}\\\gamma\in\M_1\cup\M_2\cup\{0\}\\\beta+\gamma=\alpha_s}}
    \langle\beta\rangle(-2)^{\chi_{\{1\}}(\gamma\cdot e_3)}\Biggr)x_{\ve{\alpha}},
 \end{align*}
 where $\chi_{\{1\}}$ denotes the characteristic function of $\{1\}$.

 For $\alpha=(a,b,c)\in\M^{\G}$, put
 \[
  C_{\alpha}
  =\sum_{\substack{\beta\in\M\cup\{0\}\\\gamma\in\M_1\cup\M_2\cup\{0\}\\\beta+\gamma=\alpha}}
  \langle\beta\rangle(-2)^{\chi_{\{1\}}(\gamma\cdot e_3)}
 \]
 for simplicity.
 If $c=0$, then
 \[
  C_{\alpha}=\langle\alpha\rangle=\binom{a+b}{0}=1.
 \]
 If $c=1$, then
 \begin{align*}
  C_{\alpha}
  &=\langle\alpha\rangle-2\sum_{k=0}^{\min\{a,b\}}\langle\alpha-(k,k,1)\rangle
   =\binom{a+b+1}{1}-2\sum_{k=0}^{\min\{a,b\}}\binom{a+b-2k}{0}\\
  &=(a+b+1)-2(\min\{a,b\}+1)
   =\begin{cases}
     0&\text{if $\lvert a-b\rvert=1$};\\
     -1&\text{if $a=b$}.
    \end{cases}
 \end{align*}
 If $c\ge2$, then, writing $s=a+b+c$, we have
 \begin{align*}
  C_{\alpha}
  &=\langle\alpha\rangle
    -2\sum_{k=0}^{\min\{a,b\}}\langle\alpha-(k,k,1)\rangle
    +\sum_{k=0}^{\min\{a,b\}}\langle\alpha-(k,k,2)\rangle\\
  &=\binom{s}{c}-2\sum_{k=0}^{\min\{a,b\}}\binom{s-2k-1}{c-1}+\sum_{k=0}^{\min\{a,b\}}\binom{s-2k-2}{c-2}\\
  &=\sum_{k=0}^{\min\{a,b\}}\Biggl(\binom{s-2k}{c}-\binom{s-2k-2}{c}\Biggr)
    -2\sum_{k=0}^{\min\{a,b\}}\binom{s-2k-1}{c-1}
    +\sum_{k=0}^{\min\{a,b\}}\binom{s-2k-2}{c-2}
 \end{align*}
 because $s-2\min\{a,b\}-2<c$; therefore $C_{\alpha}=0$ because
 \begin{align*}
  &\binom{s-2k}{c}-\binom{s-2k-2}{c}-2\binom{s-2k-1}{c-1}+\binom{s-2k-2}{c-2}\\
  &\qquad=\Biggl(\binom{s-2k-1}{c}+\binom{s-2k-1}{c-1}\Biggr)
   -\Biggl(\binom{s-2k-1}{c}-\binom{s-2k-2}{c-1}\Biggr)\\
  &\qquad\qquad-2\binom{s-2k-1}{c-1}+\Biggl(\binom{s-2k-1}{c-1}-\binom{s-2k-2}{c-1}\Biggr)\\
  &\qquad=0.
 \end{align*}
 In summary, we have
 \[
  C_{\alpha}=
  \begin{cases}
   (-1)^{\alpha\cdot e_3}&\text{if $\alpha\in\N$};\\
   0&\text{if $\alpha\notin\N$}.
  \end{cases}
 \]

 It follows that
 \begin{align*}
  &\sum_{i,j,p,q,u_1,\ldots,u_p,v_1,\ldots,v_q}
   (-2)^pd\bigl(x_{e_3}^i\sha(x_{e_1}x_{e_2})^j\bigr)*
   (x_{(u_1,u_1,1)}\cdots x_{(u_p,u_p,1)}\sha x_{(v_1,v_1,2)}\cdots x_{(v_q,v_q,2)})\\
  &\qquad\qquad=\sum_{\ve{\alpha}\in\I_{\N}^{\G}}(-1)^{(\wt\ve{\alpha})\cdot e_3}x_{\ve{\alpha}}.
 \end{align*}
\end{proof}

Lemmas~\ref{lem:RHS_final} and \ref{lem:LHS_final} imply Theorem~\ref{thm:in_formal},
thereby establishing our main theorem.

\end{document}